\documentclass[12pt]{article}%{amsart}
\usepackage{amsmath}
\usepackage{amssymb}
\usepackage{amsthm}
\usepackage{amscd}
\usepackage{amsxtra}
\usepackage{verbatim}
\usepackage{xcolor}
\usepackage{color}
\usepackage{enumerate}
\usepackage{mathrsfs}
\usepackage[all]{xy} 

\usepackage{array,arydshln}
\usepackage{bm}

\allowdisplaybreaks

\numberwithin{equation}{section}

\definecolor{dblue}{rgb}{0,0,0.45}
\definecolor{red}{rgb}{0.7,0,0}

 \RequirePackage{geometry}
\newtheorem{theorem}{Theorem}[section]
\newtheorem{lemma}[theorem]{Lemma}
\newtheorem*{lemma*}{Lemma}

\newtheorem{proposition}[theorem]{Proposition}

\theoremstyle{definition}

\newtheorem{remark}[theorem]{Remark}

\theoremstyle{remark}

\newcommand{\cA}{{\mathcal A}}

\newcommand{\cH}{{\mathcal H}}

\newcommand{\cK}{{\mathcal K}}
\newcommand{\cL}{{\mathcal L}}

\newcommand{\cV}{{\mathcal V}}
\newcommand{\cW}{{\mathcal W}}
\newcommand{\cX}{{\mathcal X}}
\newcommand{\cY}{{\mathcal Y}}

\newcommand{\la}{\langle}
\newcommand{\ra}{\rangle}
\newcommand{\nn}{\nonumber}

\date{}
\begin{document}

\title{Positivity of the density for rough differential equations}
\author{  Yuzuru \textsc{Inahama} and Bin \textsc{Pei}
%\footnote{ }
}
\maketitle

\begin{abstract}
Due to recent developments of Malliavin calculus 
for rough differential equations, it is now known that,
under natural assumptions, the law of 
 a unique solution at a fixed time has a smooth density function. 
Therefore, it is quite natural to ask whether or when 
the density is strictly positive.
In this paper we study this problem from the viewpoint of 
Aida-Kusuoka-Stroock's general theory.
\end{abstract}

\section{Introduction}
In this paper we study the following rough differential equation
(RDE) driven by a Gaussian rough path ${\bf w}$:
\begin{equation}\label{rde.intro}
dy_t = \sum_{i=1}^d  V_i ( y_t) dw_t^i + V_0 ( y_t) dt
\qquad
\qquad
\mbox{with \quad $y_0=a \in {\mathbb R}^e$.}
\end{equation}
Here, $V_i~(0\le i\le d)$ are nice vector fields on ${\mathbb R}^e$.
The main example of Gaussian rough paths we have in  mind
is fractional Brownian rough path (fBRP) with Hurst parameter 
$H \in (1/4, 1/2]$.
We are interested in the law of $y_t$ for a fixed time $t$. 
In particular, we take up the problem of positivity of 
the density  of the law (when the density with respect to the 
Lebesgue measure exists).

Let us review results for a usual stochastic differential equation
(which coincides with \eqref{rde.intro} driven by 
Brownian rough path).
Thanks to Malliavin calculus, the law of $y_t$ has a
smooth density 
if $V_i~(0\le i\le d)$ satisfy H\"ormander's condition
at the starting point $a$.
Positivity of the density was first studied by 
Ben Arous and L\'eandre \cite{bl2}.
Later Aida, Kusuoka and Stroock \cite{aks} generalized 
the positivity theorem for very general Wiener functionals.
(The key theorem is \cite[Theorem 2.8]{aks}.
Fortunately, its proof is not long.)

Malliavin calculus for RDE
\eqref{rde.intro} was established by \cite{chlt, ina2}.
For a class of Gaussian rough paths including 
fBRP with Hurst parameter $H \in (1/4, 1/2]$,
it is now known that the law of $y_t$ has a smooth density 
under H\"ormander's condition at the starting point $a$.
Therefore, it is quite natural to ask whether or when 
the density is strictly positive.
Concerning this problem, there are two results \cite{bnot, got2}, 
in both of which ${\bf w}$ is
 fBRP with Hurst parameter $H \in (1/4, 1/2]$.
In these works, they proved everywhere-positivity of the density 
under the uniform ellipticity or the uniform H\"ormander condition
along the lines of \cite{bl2}.

The aim of this paper is to look at this problem from 
the viewpoint of Aida-Kusuoka-Stroock's general theory. 
It is true that these two existing results are quite nice,
but it is also true that 
we can still improve them by using this theory.

The organization of this paper is as follows.
In Section \ref{sec.aks} we review 
Aida-Kusuoka-Stroock's theory on an abstract Wiener space.
In Section \ref{sec.det_RP} we collect some known
deterministic results from rough path theory,
which will be used in the proof of our main result.
Following \cite[Section 15]{fvbk},
we recall basic results on Gaussian rough paths 
in Section \ref{sec.grp}.
Section \ref{sec.core} is the core of this paper.
We prove that Lyons-It\^o map, i.e. the solution map 
for RDE \eqref{rde.intro} is twice $\cK$-differentiable
in the sense of \cite{aks}.
(See Proposition \ref{pr.Kdiff}. This is the key point in proving the positivity.)
In Section \ref{sec.cmpr} we present our main result 
in Theorem \ref{thm.pstv} and compare it with the
preceding results in Remark \ref{rem.kikaku}.

Throughout this paper the following notation will be used.
We work on the time interval $[0,T]$, 
where $T \in (0,\infty)$ is arbitrary but fixed.
For a Banach space $\cX$, 
the set of $\cX$-valued continuous paths over $[0,T]$
is denoted by $C(\cX)$.
For $a \in \cX$, the subset of continuous paths
that start at $a$ is $C_a (\cX)$.
For $p \ge 1$, set of $\cX$-valued continuous paths over $[0,T]$
with finite $p$-variation
is denoted by $C^{p -{\rm var}} (\cX)$.
In a similar way, $C_a^{p -{\rm var}} (\cX)$ is defined.
Let $f \colon U \to \cY$, where $U$ is an open subset of $\cX$
and $\cY$ is another Banach space.
For $0 \le k <\infty$,  $f$ is said to be of $C_b^k$ on $U$ if
$f$ is a bounded $C^k$-map from $U$ to $\cY$ 
whose derivatives up to order $k$ are all bounded.
The set of such $f$ is denoted by $C_b^k(U, \cY)$.
The $C_b^k$-norm of $f$ is given by
 $\sum_{i=0}^k \sup_{x \in U}\| D^i f (x) \|_{\cY}$.
If $f$ is of $C_b^k$ for all $0 \le k <\infty$,  
$f$ is said to be of $C_b^\infty$ on $U$.

%%%%%%%%%%%%%%%%%
%\newpage
\section{Review of abstract Wiener space
and $\cK$-regularity}
\label{sec.aks}

In this section we recall Aida-Kusuoka-Stroock's 
result on the positivity of the density
for non-degenerate Wiener functionals (see \cite{aks}).

In this section,  $(\cW, \cH, \mu)$ is 
an abstract Wiener space in the sense of \cite{gross1}. That is, ~$(\cW, \|\cdot\|_{\cW})$~is a separable Banach space, ~$(\cH, \|\cdot\|_{\cH})$~ is a separable Hilbert space, $\cH$~is a dense subspace of~$\cW$~and the inclusion map is continuous, and~$\mu$~is the (necessarily unique) probability measure on~$(\cW, \mathcal{B}_{\cW})$~with the property that
\begin{equation}\label{abspro}
\int_{\cW}\exp\Bigl(\sqrt{-1}_{\cW^*}\langle \lambda, w \rangle_{\cW}\Bigr)\mu (d w)=\exp\Bigl(-\frac{1}{2}\|\lambda\|^2_{\cH}\Bigr),
\qquad
\qquad
\lambda \in \cW^* \subset \cH^*,
\end{equation}
where we have used the fact that $\cW^*$ becomes a dense subspace of $\cH$ when we make the natural identification between $\cH^*$ and $\cH$ itself. 
Hence, $\cW^* \hookrightarrow \cH^* =\cH \hookrightarrow\cW$
and both inclusions are continuous and dense.
We denote by $\{ \langle k, \bullet \rangle \colon k\in \cH\}$ 
the family of centered Gaussian random variable 
defined on $\cW$ indexed by $\cH$
(i.e. the homogeneous Wiener chaos of order $1$).
If $\langle k, \bullet \rangle_{\cH} \in \cH^*$ 
extends to an element of $\cW^*$,
then the extension coincides with the  random 
variable $\langle k, \bullet \rangle$.
We also denote by $\tau_k \colon \cW \to \cW$ the translation 
$\tau_k (w) =w+k$.

For a finite dimensional subspace $K$ of $\cH$,   
$P_K\colon \cH \to K$ stands for the orthogonal projection
and we write $P_K^\perp = {\rm Id}_\cH - P_K$. 
This projection naturally extends to $\bar{P}_K\colon \cW \to K$
as follows:
\[
\bar{P}_K (w) = \sum_{i=1}^{\dim K} \langle e_i, w \rangle e_i,
\]
where $\{ e_i\}_{i=1}^{\dim K}$ is an orthonormal basis of $K$.
(This right hand side
 is independent of the choice of $\{ e_i\}$.)
We set $\bar{P}_K^\perp =  {\rm Id}_\cW - \bar{P}_K$.

Now we recall the definitions of $\cK$-continuity,
$\cK$-regularity, uniformly $\cK$-regularity, 
and $l$-times $\cK$-regular differentiability, 
which were first introduced by \cite{aks}. 
Note that in these definitions, functions and maps on $\cW$
are everywhere-defined ones (not equivalence classes 
with respect to $\mu$).

Assume that $\cK=\{K_n\}^{\infty}_{n=1}$ is a non-decreasing, countable exhaustion of $\cH$ by finite dimensional subspaces, 
that is, $K_n \subset K_{n+1}$ for all $n$
and $\cup_{n=1}^\infty K_n$ is dense in $\cH$.
Set $P_n=P_{K_n}$, define $\bar P_n,$, $\bar P_n^{\perp}$ 
accordingly. We say that a map $F$ from $\cW$ into 
a Polish space $(E, \rho_E)$ is {\it $\cK$-continuous} 
if it is measurable and, 
for each $n\in\mathbb{N} :=\{1,2, \ldots\}$, 
there is a measurable map $F_n\colon\cW\times K_n\longmapsto E$ with the properties that $F\circ \tau_k=F_n(\cdot,k)$ (a.s. $\mu$)
 for each $k\in K_n$ and $k\in K_n \longmapsto F_n(w,k)\in E$ is continuous for each $w \in \cW$. Given a $\cK$-continuous 
 map $F$,  we set
\begin{eqnarray}
 F_n^{\perp}(w,k) = F_n(w,-\bar P_n (w) +k)
\qquad {\rm for}~n\in \mathbb{N} ~{\rm and}~k\in K_n.
\end{eqnarray}

Given a measurable map $F \colon\cW\rightarrow E$, we will say that $F$ is {\it $\cK$-regular} if $F$ is $\cK$-continuous and there is a continuous map
$\tilde {F} \colon\cH \rightarrow E$ such that
\begin{eqnarray}\label{rhoe}
 \lim_{n \rightarrow \infty} \mu \bigg(\bigg\{w\colon\rho_E(\tilde{F}\circ \bar{P}_n(w),F (w))\vee
  \rho_E(\tilde{F}(h),F_n^{\perp}(w, P_n (h)))
 \geq \epsilon\bigg\}\bigg)=0
 \end{eqnarray}
holds for every $\epsilon>0$ and $h\in \cH$. 
In this case $\tilde {F}$ will be called 
a $\cK$-regularization of $F$.

If  $F$ is a map from $\cW$ into a Polish space $E$, we will say that it is {\it uniformly $\cK$-regular} if it is $\cK$-regular and (\ref{rhoe}) can be replaced by the condition that
\begin{eqnarray}
 \lim_{n \rightarrow \infty} \mu \bigg(\bigg\{w &\colon& \sup_{k\in K_m, \|k\|_{\cH}\leq r}\rho_E(\tilde{F}(\bar{P}_n(w)+k),F_{m\vee n}(w,k)) \cr
 &&\vee \rho_E(\tilde{F}(h+k), 
 F^{\perp}_{m\vee n}(w, P_n(h)+k)) \geq \epsilon\bigg\}\bigg)=0
 \label{rhoe2}
  \end{eqnarray}
for every $m\in \mathbb{N}, r>0, \epsilon>0$ and $h\in \cH$. 

Let $E$ be a separable Banach space 
and $F$ be a map from $\cW$ into $E$.
Given $l\in \mathbb{N}$ we will say that $F$ is 
{\it $l$-times $\cK$-regularly differentiable} 
if $F$ is uniformly $\cK$-regular, $F_n (w,\cdot)$ is $l$-times continuously Fr\'{e}chet differentiable on $K_n$ for each $n\in \mathbb{N}$ and $w \in \cW$, $\tilde{F}$ is $l$-times continuously 
Fr\'{e}chet differentiable 
on $\cH$, and  (\ref{rhoe2}) can be replaced by the condition that
\begin{eqnarray}\label{rhoe3}
 \lim_{n \rightarrow \infty} \mu \bigg(\bigg\{w &\colon& \|\tilde{F}(\bar{P}_n(w)+\bullet)-F_{m\vee n}(w,\bullet))\|_{C_{b}^l (B_{K_m}(0,r), E)} \cr
 &&\vee \|\tilde{F}(h+\bullet)-F^{\perp}_{n}(w, P_n(h)+\bullet)\|_{C_{b}^l (B_{K_m}(0,r), E)} \geq \epsilon\bigg\}\bigg)=0
 \end{eqnarray}
for every $m\in \mathbb{N}, r>0, \epsilon >0$ and $h\in \cH$. 
Here, $B_{K_m}(0,r)=\{ k \in K_m \colon \|k\|_{\cH}< r \}$.

The following theorem is \cite[Theorem 2.8]{aks}
and plays a key role in this paper.
It is a quite general result on the positivity of the density 
function of the law of a non-degenerate Wiener functional.
(In this theorem, any choice of $\tilde{F}$ and $\tilde{G}$ will do.
If $G \equiv 1$, then $f$ is the density 
function of the law of $F$ on ${\mathbb R}^e$.)
\begin{theorem}\label{aks>0}
Let $F\colon\cW\rightarrow \mathbb{R}^e$, $e \in \mathbb{N}$, and 
$G\colon\cW\rightarrow [0,+\infty)$ be functions which are infinitely differentiable in the sense of the Malliavin calculus for $(\cW,\cH,\mu)$, and assume that the associated Malliavin covariance matrix 
$DF  \cdot D F ^T = \{\langle DF^i (w), DF^j (w)\rangle_{\cH} \}_{1 \le i,j \le e}$ 
is non-degenerate in the Malliavin sense, namely,
\begin{equation}\label{cond.nondeg}
(\det (DF \cdot D F ^T))^{-1}\in
 \bigcap_{p\in[1,+\infty)} L^p(\cW; \mu).
 \end{equation}
Here, $D$ stands for the $\cH$-derivative. 
Then, there exists a unique non-negative function
$f\in C^{\infty}(\mathbb{R}^e, \mathbb R)$ with the properties that $f$ has rapidly decreasing derivatives of all orders and 
\[
\int_{\cW}(\phi\circ F) (w) G (w)
\mu (dw)
=\int_{\mathbb{R}^e}\phi (x)  f (x) dx, 
\qquad
\phi \in C_b(\mathbb{R}^e, \mathbb R).
\]
Assume further that $F$ is twice $\cK$-regularly differentiable 
and $G$ is $\cK$-regular with their $\cK$-regularizations 
$\tilde{F}$ and $\tilde{G}$, respectively. 
Then,  for $y\in \mathbb{R}^e$, the following  are equivalent:
\begin{itemize}
\item
$f(y)>0$.
\item
There exists $h \in \cH$ such that 
$D\tilde{F}(h)\colon \cH\rightarrow \mathbb{R}^e$ has rank $e$, 
$\tilde{F}(h)=y$ and 
$\tilde{G}(h)>0$. 
\end{itemize}
\end{theorem}

%%%%%%%%%%%%%%%%
%\newpage
%%%%%%%%%%%%%%%%
\section{Some deterministic results from rough path theory}
\label{sec.det_RP}

In this and the next sections we recall basic results 
on rough paths and RDEs.
In this section we summarize deterministic facts 
which will be used later.
In what follows we will always
assume $2 \le p < 4$, $1 \le q <2$.
The integer part of $p$ is denoted by $[p]$.

The geometric rough path space with $p$-variation topology 
over ${\mathbb R}^d$ is denoted by 
$G\Omega_p ({\mathbb R}^d)$.
An element of $G\Omega_p ({\mathbb R}^d)$ is denoted by
${\bf x} = ({\bf x}^1, \ldots, {\bf x}^{[p]})
= ({\bf x}^1_{s,t}, \ldots, {\bf x}^{[p]}_{s,t})_{0 \le s \le t \le T}$. 
Recall that the $p$-variation topology is induced by
the following variation norms:
\[
\|{\bf x}^i \|_{p/i -{\rm var}}
:=
\sup_{0=t_0<\cdots<t_K=T}
				\Bigl(
					\sum_{k=1}^K
						|{\bf x}^i_{t_{k-1},t_k}|^{p/i}
				\Bigr)^{i/p},
				\qquad 
				 1 \le i \le [p].
				\]
Here, $\{0=t_0<\cdots<t_K=T\}$ runs over all finite 
partition of $[0,T]$.
For more details, see \cite{fvbk} for example.

Now we introduce an RDE.
Recall that an RDE itself is deterministic.
In this work we only treat the first level path of 
the solution of an RDE and
therefore we simply call it a solution of the RDE.

Let $V_{i}\colon {\mathbb R}^e \to {\mathbb R}^e$ be a vector field on ${\mathbb R}^e$
with sufficient regularity ($0 \le i \le d$)
and 
let $G\Omega_p ({\mathbb R}^d)$ be the geometric rough path space over ${\mathbb R}^d$
with $p$-variation topology ($2 \le p < 4$).
We consider the following 
RDE driven by ${\bf x} \in G\Omega_p ({\mathbb R}^d)$:
\begin{equation}\label{rde.def}
dy_t = \sum_{i=1}^d  V_i ( y_t) dx_t^i + V_0 ( y_t) dt
\qquad
\qquad
\mbox{with \quad $y_0=a \in {\mathbb R}^e$.}
\end{equation}
If $V_i$'s are of $C_b^{[p] +1}$, then a unique solution 
$y= y({\bf x})$ exists,
which is denoted by $\Phi ({\bf x})$.
Moreover, $\Phi\colon G\Omega_p ({\mathbb R}^d)
\to C_a^{p -{\rm var}} ({\mathbb R}^e)$ is locally 
Lipschitz continuous, that is,
Lipschitz continuous on every bounded subset of 
$G\Omega_p ({\mathbb R}^d)$.
This map is called Lyons-It\^o map (associated with $V_i$'s).
We remark that $\Phi$ actually takes values in 
$C_a^{p -{\rm var}} ({\mathbb R}^e) \cap C^{0,p -{\rm var}} ({\mathbb R}^e)$.
Here, $C^{0,p -{\rm var}} ({\mathbb R}^e)$ is a separable 
Banach subspace 
of $C^{p -{\rm var}} ({\mathbb R}^e)$ defined as the closure of the set of 
all ${\mathbb R}^e$-valued $C^1$-path.

Let $1 \le q <2$. 
For  $h \in C_0^{q -{\rm var}} ({\mathbb R}^d)$, 
the ordinary differential equation in the sense 
of Young integral (Young ODE) that corresponds to 
\eqref{rde.def} is given as follows:
\begin{equation}\label{ode.def}
dy_t = \sum_{i=1}^d  V_i ( y_t) dh_t^i + V_0 ( y_t) dt
\qquad
\qquad
\mbox{with \quad $y_0=a \in {\mathbb R}^e$.}
\end{equation}
If $V_i$'s are of $C_b^{2}$, then a unique solution 
$y$ exists,
which is denoted by $\Psi (h)$.
Under the same condition, 
$\Psi\colon C_0^{q -{\rm var}} ({\mathbb R}^d) \to 
C_a^{q -{\rm var}} ({\mathbb R}^e)$ is locally Lipschitz continuous.

\bigskip

Using Young integration, we define ${\bf h} \in G\Omega_p ({\mathbb R}^d)$ by 
\[
{\bf h}^i_{s,t} = \int_{s \le u_1 \le \cdots \le u_i \le t}   
dh_{u_1} \otimes \cdots  \otimes dh_{u_i},
\qquad
0\le s \le t \le T, \,\, 1 \le i \le [p].
\]
We often write ${\bf h} = \cL (h)$, too.
The lift map $\cL\colon C_0^{q -{\rm var}} ({\mathbb R}^d)
\to G\Omega_p ({\mathbb R}^d)$ is also 
locally Lipschitz continuous, injective
and its image is dense.
As one can easily guess, $\Psi (h) = \Phi ({\bf h})$ holds. 
In this sense, RDE \eqref{rde.def} generalizes 
Young ODE \eqref{ode.def}.

Under the condition that $1/p +1/q >1$, 
there exists a translation on the geometric rough path space
which is compatible 
with the usual translation on the usual path space. 
It is called Young translation 
$T \colon G\Omega_p ({\mathbb R}^d) \times 
C_0^{q -{\rm var}} ({\mathbb R}^d) \to G\Omega_p ({\mathbb R}^d)$ 
and is characterized as a unique continuous map that satisfies 
$
T_h ({\bf k}) = \cL (h+k)$ for every 
$h, k \in C_0^{q -{\rm var}} ({\mathbb R}^d)$.
$T_h ({\bf x})= T_h {\bf x}$ should be viewed as the translation of 
${\bf x}$ by $h$.
(See \cite[Subsection 9.4]{fvbk} for example.)

Next, let us see how  derivatives of $\Psi$ look like.
For brevity, we write $\sigma = [V_1, \ldots, V_d]$ and $b =V_0$ 
and view them as an $e \times d$ matrix-valued 
and an ${\mathbb R}^e$-valued function, respectively.
Then, \eqref{ode.def} simply reads
$dy_t = \sigma ( y_t) dh_t + b ( y_t) dt$.
By formal differentiation of
\eqref{ode.def} in the direction 
of $l \in C_0^{q -{\rm var}} ({\mathbb R}^d)$, 
$D_l y_t$ should satisfy the following Young ODE
(if it exists):
\begin{equation}
d \xi_t^{[1]}
= 
\nabla \sigma (y_t ) \la  \xi_t^{[1]}, dh_t \ra
+ 
\nabla b(y_t ) \la  \xi_t^{[1]}\ra dt
+
\sigma (y_t ) dl_t
\quad
\mbox{with \quad $\xi_0^{[1]}  =0\in {\mathbb R}^e$.}
\label{heu_1.eq}
\end{equation}
Here, $D_l$ stands for the directional derivative on 
$C_0^{q -{\rm var}} ({\mathbb R}^d)$ in the direction $l$
and 
$\nabla$ stands for the standard gradient on ${\mathbb R}^e$.
In  a similar way, $D^2_{l,l} y_t$ should satisfy the 
following Young ODE (if it exists):
\begin{align}
d \xi_t^{[2]}
&= 
\nabla \sigma ( y_t) \la \xi_t^{[2]}, dh_t \ra 
+
\nabla b ( y_t) \la \xi_t^{[2]} \ra  dt
+
\nabla^2 \sigma ( y_t) \la \xi_t^{[1]},  \xi_t^{[1]}, dh_t\ra
\nn\\
&
+
2\nabla \sigma ( y_t) \la \xi_t^{[1]},dl_t\ra 
+
\nabla^2 b ( y_t) \la \xi_t^{[1]},  \xi_t^{[1]}\ra dt
\quad
\mbox{with  \quad $\xi_0^{[2]}=0\in {\mathbb R}^e$.}
\label{heu_2.eq}
\end{align}

If $V_i$'s are of $C_b^{4}$,  the following facts are known
to hold (see \cite{ll, ina1} for example).
The system of Young ODEs 
\eqref{ode.def}, \eqref{heu_1.eq} and  \eqref{heu_2.eq}
has a unique global solution for every $h$ and $l$. 
(We write $\xi_t^{[1]}(h,l)$ and $\xi_t^{[2]}(h,l)$ when necessary.)
The mapping 
\[
C_0^{q -{\rm var}} ({\mathbb R}^d\oplus {\mathbb R}^d)\ni
(h,l) \mapsto (y(h), \xi^{[1]}(h,l), \xi^{[2]}(h,l))
\in
C_0^{q -{\rm var}} (({\mathbb R}^e)^{\oplus 3} 
)
\]
is locally Lipschitz continuous.
Moreover, 
 $\Psi$ is of Fr\'echet-$C^2$ and 
$D_l \Psi (h)$ and $D^2_{l,l} \Psi (h)$ coincide with 
$\xi^{[1]}(h,l)$ and $\xi^{[2]}(h,l)$, respectively.

Now we go back to RDEs. 
We assume that $V_i$'s are of $C_b^{[p]+3}$.
The RDEs driven by ${\bf x}\in G\Omega_p ({\mathbb R}^d)$ 
and $l\in C_0^{q -{\rm var}} ({\mathbb R}^d)$
which correspond to \eqref{heu_1.eq}--\eqref{heu_2.eq}
are given as follows:
\begin{align}
d \xi_t^{[1]}
&= 
\nabla \sigma (y_t ) \la  \xi_t^{[1]}, dx_t \ra
+ 
\nabla b(y_t ) \la  \xi_t^{[1]}\ra dt
+
\sigma (y_t ) dl_t
\quad
\mbox{with \quad $\xi_0^{[1]}  =0\in {\mathbb R}^e$.}
\label{heu_3.eq}
\\
d \xi_t^{[2]}
&= 
\nabla \sigma ( y_t) \la \xi_t^{[2]}, dx_t \ra 
+
\nabla b ( y_t) \la \xi_t^{[2]} \ra  dt
+
\nabla^2 \sigma ( y_t) \la \xi_t^{[1]},  \xi_t^{[1]}, dx_t\ra
\nn\\
&
+
2\nabla \sigma ( y_t) \la \xi_t^{[1]},dl_t\ra 
+
\nabla^2 b ( y_t) \la \xi_t^{[1]},  \xi_t^{[1]}\ra dt
\quad
\mbox{with  \quad $\xi_0^{[2]}=0\in {\mathbb R}^e$.}
\label{heu_4.eq}
\end{align}
It is known that the system of RDEs 
\eqref{rde.def}, \eqref{heu_3.eq} and \eqref{heu_4.eq}
has a unique global solution for every $({\bf x}, l)$.
(See \cite{ina2} or \cite[Section 10.7]{fvbk}  for example.)
When we specify the driver, we will write  
$\xi_t^{[j]} = \xi_t^{[j]} ({\bf x}, l)$, $j=1,2$.
For $h\in C_0^{q -{\rm var}} ({\mathbb R}^d)$, we have
$\xi^{[j]} (h,l)= \xi^{[j]} ({\mathcal L} (h), l)$, $j=1,2$.
Since explosion never happens, Lyons' continuity theorem 
still holds for this system of RDEs,
namely, the following map is locally Lipschitz continuous:
\begin{equation}\label{eq.der.Lip.conti}
G\Omega_p ({\mathbb R}^d) \times C_0^{q -{\rm var}} ({\mathbb R}^d)
\ni ({\bf x}, l)
\mapsto 
(y({\bf x}), \xi^{[1]}({\bf x},l), \xi^{[2]}({\bf x},l))
\in
C^{0, p -{\rm var}} (({\mathbb R}^e)^{\oplus 3} ).
\end{equation}
This property plays a key role in Section \ref{sec.core}.

%%%%%%%%%%%%%%%%%
%\newpage
%%%%%%%%%%%%%%%%
\section{Gaussian rough path}
\label{sec.grp}

We introduce a stochastic process.
Let $(w_t)_{ 0 \le t \le T} = (w_t^1, \ldots, w_t^d)_{ 0 \le t \le T}$
be a centered, continuous, $d$-dimensional Gaussian process 
with i.i.d. components which start at $0$.
We denote its law by $\mu$ and its Cameron-Martin space 
by $\cH$.
Then, $(\cW, \cH, \mu)$ 
becomes an abstract Wiener space, where $\cW$ is the closure  
of $\cH$
with respect to the usual sup-norm
in $C_0 ({\mathbb R}^d)
:=\{x\colon [0,T] \to{\mathbb R}^d \colon \mbox{continuous and $x_0 =0$}  \}$.
When $d=1$, we write $(\cW_1, \cH_1, \mu_1)$.
Obviously, $\cW = (\cW_1)^{\oplus d}$ (the direct sum),
$\cH = (\cH_1)^{\oplus d}$ (the orthogonal sum),
$\mu = (\mu_1)^{\otimes d}$ (the product measure).

Let $R(s,t) = {\mathbb E} [w^1_s w^1_t]$ be the covariance function.
For the rest of this paper we will assume the following condition:
\begin{equation}\label{cond.191015}
\mbox{
$R (s,t)$
is of finite 2D $\rho$-variation for some $\rho \in [1, 2)$,
}
\end{equation}
(For the definition of 2D $\rho$-variation, see 
\cite[Section 15.1]{fvbk}.)  
Under \eqref{cond.191015} many facts were shown. 
Some of them are as follows.
First, $\cH$ is continuously embedded in 
$C_0^{\rho -{\rm var}} ({\mathbb R}^d)$
(see \cite[Proposition 15.7]{fvbk}). 
This implies that we can use Young integration for 
$h \in \cH$ since $\rho <2$.
Second, the natural lift ${\bf w}$ of $w$ exist
as a $G \Omega_p ({\mathbb R}^d)$-valued random variable
for $p \in (2\rho, 4)$.
All known reasonable 
rough path lift of $w$ coincides with ${\bf w}$.
Those include the limit of piecewise linear,
mollifier, and Karhunen-Lo\'eve 
approximations of $w$ 
(see \cite[Theorem 15.33, Definition 15.34]{fvbk}).
In this work we will only use Karhunen-Lo\'eve  approximations. 
We will recall it in the next paragraph.

Let $\{ h_i \}_{i=1}^{\infty}$ be an orthonormal basis of $\cH_1$
such that $\langle h_i, \bullet \rangle \in (\cW_1)^*$
and $\{ {\bf e}_j \}_{j=1}^{d}$ be the canonical 
orthonormal basis of ${\mathbb R}^d$.
Set $h_{i, j} = h_i   {\bf e}_j \in \cH$, then 
$\{ h_{i,j} \}_{1\le i <\infty, 1 \le j \le d}$ forms 
an orthonormal basis of $\cH$ and 
$\langle h_{i,j}, \bullet \rangle \in \cW^*$.
As is well-known, 
$\{ \langle h_{i,j}, \bullet \rangle \}_{1\le i <\infty, 1 \le j \le d}$ 
is an i.i.d. defined on $(\cW, \cH, \mu)$
with the law of $\langle h_{i,j}, \bullet \rangle$ being 
the standard normal distribution.

Define Gaussian processes $w^{[N]}$ and $w^{*N}$
for $w \in \cW$ and  $N \ge 1$ by
\begin{equation}\nn
w^{[N]} = \sum_{i=1}^N \sum_{j=1}^d 
\langle h_{i,j}, w \rangle h_{i,j},
\qquad
w^{*N}= w - w^{[N]}.
\end{equation}
Note that these are everywhere-defined 
random variables defined on $\cW$ 
taking values in $\cH$ and $\cW$, respectively.
Since the former is $\cH$-valued, the lift 
${\bf w}^{[N]} :=\cL (w^{[N]})$ exists for every $w$.
The
covariance of $w^{*N}$ also satisfies \eqref{cond.191015}
with the same $\rho$ (see \cite[p. 438]{fvbk}).
Therefore, the natural lift of ${\bf w}^{*N}$ of $w^{*N}$ exists, too.

Now we recall Karhunen-Lo\'eve  approximations
for Gaussian rough paths.
The following is (a special case of) 
\cite[Theorem 15.47]{fvbk}.
\begin{proposition}\label{pr.KL}
Let the notation be as above and assume \eqref{cond.191015}
and $p \in (2\rho, 4)$.
Then, we have the following:
\begin{itemize}
\item
There exists a positive constant $\eta$ independent of $N$
such that
\[
\sup_{N \ge 1} {\mathbb E} \bigl[ \exp \bigl( \eta 
\sum_{i=1}^{[p]}  \| ({\bf w}^{[N]})^i \|^{2/i}_{p/i -{\rm var}}
\bigr)
\bigr] <\infty.
\]
\item
For every $r \in [1,\infty)$ and $1\le i \le [p]$,
\[
 \| ({\bf w}^{[N]})^i - {\bf w}^i\|_{p/i -{\rm var}} \to 0
\qquad
\mbox{in $L^r(\mu)$ as $n\to \infty$.}
\]
\item
For every $r \in [1,\infty)$ and $1\le i \le [p]$,
\[
 \| ({\bf w}^{*N})^i \|_{p/i -{\rm var}} \to 0
\qquad
\mbox{in $L^r(\mu)$ as $n\to \infty$.}
\]
\end{itemize}
\end{proposition}

We further assume the following condition,
which is called complementary Young regularity in 
\cite[Condition 15.56]{fvbk}.
When we work under this condition, we pick 
(any) $p$ and $q$ as in the statement of this condition.
\\
\\
{\bf (CYR)}~
In additon to \eqref{cond.191015},
there exist $p \in  (2\rho, 4)$ and $q \in [1,2)$ such that (1)~
$1/p +1/q >1$ 
and (2)~ ${\cal H}$ is continuously embeded in $C_0^{q -{\rm var}} ({\mathbb R}^d)$,
the space of ${\mathbb R}^d$-valued continuous paths
of finite $q$-variation that start at $0$.

\medskip
\begin{remark}
Assumption {\bf (CYR)} holds if one of the following conditions holds:
\\
\noindent 
{\rm (i)}~$R$ is  of finite 2D $\rho$-variation for some $\rho \in [1,3/2)$.
\\
\noindent 
{\rm (ii)}~$w$ is fractional Brownian motion (fBM)
with Hurst parameter $H \in (1/4, 1/2]$.
\end{remark}

\begin{lemma}\label{lm.191017}
Under  {\bf (CYR)}, we have
$
{\bf w}^{*N}= T_{- w^{[N]}} ( {\bf w})$, 
almost surely. 
Here, 
$T\colon G \Omega_p ({\mathbb R}^d) \times 
C_0^{q -{\rm var}} ({\mathbb R}^d)
\to G \Omega_p ({\mathbb R}^d)$ stands for the Young translation.
\end{lemma}

\begin{proof}
For $w \in \cW$ and $k \ge 1$, we denote by 
$w (k)$ the dyadic piecewise linear approximation 
associated with $\{ l2^{-k} T\colon 0 \le l \le 2^k\}$.
By the piecewise linear approximation theorem
for Gaussian rough path \cite[Theorem 15.42]{fvbk},
we have $\cL (w(k)) \to {\bf w}$ a.s. 
and $\cL (w^{*N}(k)) = T_{-w^{[N]} (k) } \cL (w(k))
\to {\bf w}^{*N}$ a.s. as $k \to \infty$, 
taking a subsequence if necessary.
Take $\delta >0$ so small that $1/p +1/(q+\delta) >1$.
By \cite[Theorem 5.23]{fvbk}, 
we have $h_{i,j} (k)\to h_{i,j}$ in $C_0^{q+\delta -{\rm var}} ({\mathbb R}^d)$
as $k \to \infty$.
Hence, for all $w\in \cW$,
 $w^{[N]} (k)\to w^{[N]}$ in $C_0^{q+\delta -{\rm var}} ({\mathbb R}^d)$
as $k \to \infty$.
Using the continuity of $T$, we finish the proof.
\end{proof}

By Proposition \ref{pr.KL}, we can find
 a subsequense $\{N_n \}_{n=1}^\infty$
such that 
${\bf w}^{[N_n]} \to {\bf w}$ a.s. 
and ${\bf w}^{*N_n} \to {\bf 0}$ a.s. 
in 
$G \Omega_p ({\mathbb R}^d)$ as $n \to \infty$.
Here, ${\bf 0}$ stands for the zero rough path. 
(In some literature it is denoted by ${\bf 1}$.)
We take such $\{N_n \}_{n=1}^\infty$ and
set 
\[
\tilde\cA :=\{ w \in \cW \colon 
\mbox{$\{ {\bf w}^{[N_n]}\}_{n=1}^\infty$  converges
in $G \Omega_p ({\mathbb R}^d)$}\}.
\]
As a version of ${\bf w}$, we choose the following 
everywhere-defined Borel-measurable map 
${\bf L}\colon \cW \to G \Omega_p ({\mathbb R}^d)$:
\[
{\bf L} (w)=\left \{
\begin{array}{ll}
\lim_{n\to \infty} {\bf w}^{[N_n]}
& \mbox{if $w \in \tilde\cA$
} \\
{\bf 0}  &\mbox{if otherwise.}
\end{array}
\right.
\]
Since $h^{[N]} \to h$ in $\cH$ as $N \to \infty$
and we have continuous injections  
$\cH \hookrightarrow C_0^{q -{\rm var}} ({\mathbb R}^d) 
\hookrightarrow G \Omega_p ({\mathbb R}^d)$, 
we can easily see that $h\in \tilde\cA$ and 
$\cL (h) ={\bf L} (h)$
for every $h \in \cH$.

Note that if  $\{ {\bf w}^{[N_n]} = \cL (w^{[N_n]})\}$  is convergent
then so is $\{ \cL ((w+h)^{[N_n]})\}$ for every $h \in \cH$ since
\[
\cL ((w+h)^{[N_n]}) = \cL (w^{[N_n]} +h^{[N_n]})
= T_{h^{[N_n]}}  \cL(w^{[N_n]}) \to  T_h  {\bf L} (w)
\quad
\mbox{as $n \to \infty$.}
\]
This implies that 
$\tilde\cA$ is invariant under $\tau_h$
for every $h \in \cH$,
where $\tau_h (w) =w+h$,  
and that ${\bf L} (w+h)=T_h  {\bf L} (w)$
for every $w \in \tilde\cA$ and $h \in \cH$.

By Lemma \ref{lm.191017}, 
we can choose
$T_{- w^{[N]}}  {\bf L} (w)$ as an
everywhere-defined Borel-measurable version of ${\bf w}^{*N}$
and will fix this version from now on.
Then,  we can find a subsequence of $\{N_n\}$,
which will be denoted by the same symbol again, such that
the following subset 
\[
\cA :=\{ w \in \tilde\cA \colon 
\mbox{$\lim_{n \to\infty} {\bf w}^{*N_n} ={\mathbf 0}$
in $G \Omega_p ({\mathbb R}^d)$}\}.
\]
is of full measure with respect to $\mu$.
Then, $\cA$ is also invariant under $\tau_h$ for every $h \in \cH$. 
Indeed, if $w \in \cA$, then 
\[
T_{- (w+h)^{[N_n]}}  {\bf L} (w+h) =T_{h-h^{[N_n]}  } T_{- w^{[N_n]}}  {\bf L} (w)
\to T_0 {\bf 0} ={\bf 0}
\qquad
\mbox{as $n \to\infty$},
\]
which implies $w+h \in \cA$.
In particular, $\cH \subset \cA$ since $0 \in \cA$.

Choose a subsequence $\{ N_n\}$ as above and 
fix it in what follows.
Let $K_n \subset \cH$ be the finite-dimensional 
subspace spanned by $\{ h_{i,j} \}_{1\le i \le N_n, 1 \le j \le d}$
and denote by $P_n \colon \cH \to K_n$
the orthogonal projection onto this subspace. 
As is well-known, $P_n$ naturally extends to 
the bounded projection $\bar{P}_n \colon \cW \to K_n$
given explicitly by $\bar{P}_n (w) = w^{[N_n]}$.
We set 
$P_n^\perp (h) = h- P_n (h)$ for $h \in \cH$
and
$\bar{P}_n^\perp (w) = w- \bar{P}_n (w)$ for $w \in \cW$.
If we use this notation, 
${\bf w}^{*N_n} =  {\bf L} (w -w^{[N_n]}) =  {\bf L} ( P_n^\perp (w) )$ 
for all $w \in \cA$.

%%%%%%%%%%%%%%%%%%%%%%%%%%
\begin{remark}
Consider the case $\cW =C_0 ({\mathbb R}^d)$
(e.g. the case that $w$ is fBM with $1/4 <H \le 1/2$).
Then, the situation  described above can be summarized by the 
following commutative diagram:
\[
  \xymatrix{
  &  & G \Omega_p ({\mathbb R}^d)   \ar[rd]^{\Phi} & \\
 \cH  \ar@{^{(}-_>}[r]  
   & C_0^{q -{\rm var}} ({\mathbb R}^d) \ar[ru]^{\cL}\ar@{^{(}-_>}[r]
       & {\mathcal W} \ar[u]_{{\bf L}}   & C^{0,p -{\rm var}} ({\mathbb R}^e)
  }
\]
Here, all the maps except ${\bf L}$ are continuous.
Recall also that $\Psi = \Phi \circ \cL$.
(If $\cW \neq C_0 ({\mathbb R}^d)$,  
we just need to replace 
 $C_0^{q -{\rm var}} ({\mathbb R}^d)$ in this diagram by 
the closure of $\cH \cap C_0^{q -{\rm var}} ({\mathbb R}^d)$ 
with respect to the $q$-variation norm.)
\end{remark}

%%%%%%%%%%%%%%%%%%%%%%%%%%%
%\newpage

\section{Twice $\cK$-differentiability of Lyons-It\^o map}
\label{sec.core}

For $\cK=\{K_n\}_{n=1}^\infty$ as in the previous section,
the rough path lift map is $\cK$-regular
and so is the solution of an RDE driven a Gaussian rough path.
Note that $\Phi \circ {\bf L}$ equals a.s. to the solution of 
RDE \eqref{rde.def} with ${\bf x}$ being replaced by ${\bf w}$.

\begin{proposition}\label{pr.Kreg}
Let the notation be as above and
assume {\bf (CYR)}.
Then, we have the following:
\begin{itemize}
\item[{\rm (1)}]
The measurable map ${\bf L}\colon \cW \to G \Omega_p ({\mathbb R}^d)$ is uniformly $\cK$-regular.
\item[{\rm (2)}]
If, in addition, 
$V_i$ is of $C_b^{[p] +1}$ for all $0 \le i \le d$, then 
$\Phi \circ 
{\bf L}\colon \cW \to C^{0, p -{\rm var}} ({\mathbb R}^e)$ is uniformly
$\cK$-regular.
\end{itemize}
\end{proposition}

\begin{proof}
First we show (1).
We set $E^\prime = G \Omega_p ({\mathbb R}^d)$, $G={\mathbf L}$ and
$\tilde{G} = \cL\restriction_\cH\colon \cH \to E^\prime$. 
Set also 
$G_n\colon\cW\times K_n\longmapsto E^\prime$ by 
$G_n (w,k) = T_k {\mathbf L}(w)$ if $w \in \cA$
and $G_n (w,k) = {\mathbf 0}$ if $w \notin \cA$.
Then, for all 
$w \in \cA$ and $k \in \cH$, we have 
\begin{align*}
G \circ \tau_k (w) &= {\mathbf L}(w+k) = T_k  {\mathbf L}(w) = G_n (w,k),
\\
G_n^\perp (w,k) &= G_n(w,-\bar P_n (w)+k)
={\mathbf L}(w-\bar P_n (w)+k )
={\mathbf L}(\bar P_n^\perp (w)+k ) = T_k  {\mathbf w}^{*N_n}.
\end{align*}
Thanks to these explicit expressions, we may and will view
 $G_n$ and $G_n^\perp$ as
  maps from  $\cW\times \cH$ to $E^\prime$.

Let us check \eqref{rhoe2}. Take $w \in \cA$.
Note that $\tilde{G}(\bar{P}_n(w)+k) = \cL (\bar{P}_n(w)+k )
= T_k {\mathbf w}^{[N_n]}$ and 
$G_{m\vee n}(w,k) = T_k {\mathbf L}(w)$.
Since ${\mathbf w}^{[N_n]} \to  {\mathbf L}(w)$ as $n \to \infty$,
$\{ {\mathbf w}^{[N_n]} \}_n$ is bounded in $E^\prime$.
Since $T\colon E^\prime \times \cH \to E^\prime$ is locally Lipschitz continuous,
we see that 
\begin{align*}
\sup_{\|k\|_{\cH}\leq r}
\rho_{E^\prime} (\tilde{G}(\bar{P}_n(w)+k),G_{m\vee n}(w,k))
&=
\sup_{\|k\|_{\cH}\leq r}
\rho_{E^\prime} (T_k {\mathbf w}^{[N_n]},   T_k {\mathbf L}(w))
\\
&\le C_{r, w} \rho_{E^\prime} ( {\mathbf w}^{[N_n]}, {\mathbf L}(w))
\to 0
\quad
\mbox{as $n \to \infty$.}
\end{align*}
Here, 
$C_{r, w}$ is a positive constant which depends only on $r>0$
and $w \in \cA$ (and may vary from line to line).
In a similar way, since ${\mathbf w}^{*N_n}\to {\bf 0}$ as $n\to \infty$,
\begin{align*}
\sup_{\|k\|_{\cH}\leq r}
\rho_{E^\prime} (\tilde{G}(h+k),G_n^{\perp}(w, P_n (h)+k))
&=
\sup_{\|k\|_{\cH}\leq r}
\rho_{E^\prime} ( T_{k+h} {\mathbf 0},  T_{k+ P_n (h)}  {\mathbf w}^{*N_n})
\\
&\le 
C_{r, w, h} \{ \rho_{E^\prime} ( {\mathbf w}^{*N_n}, {\mathbf 0}) 
+ \| P_n (h) -h \|_{\cH} \}
\to 0
\end{align*}
as $n \to \infty$.
Here, $C_{r, w, h}$ is a positive constant 
which depends only on $r>0$, $h\in \cH$
and $w \in \cA$. 
Thus, we have shown (1).

Next we show (2).
We set $E = C^{0, p -{\rm var}} ({\mathbb R}^e)$, 
$F=\Phi \circ {\mathbf L}$ and
$\tilde{F} =\Phi \circ \cL\restriction_\cH\colon \cH \to E$. 
Note that $\tilde{F} = \Psi\restriction_\cH$.
Set also 
$F_n\colon\cW\times K_n\longmapsto E$ by 
$F_n (w,k) = \Phi (G_n (w,k))$.
Take any $w \in \cA$ and $k \in \cH$. 
It is clear that $F \circ \tau_k (w)= F_n (w,k)$.
We also have $F_n^\perp (w,k) =  \Phi(T_k  {\mathbf w}^{*N_n})$.
Again, we may and will view 
 $F_n$ and $F_n^\perp$ as maps from  $\cW\times \cH$ to $E$.

Since Lyons-It\^o map $\Phi$ is locally Lipschitz continuous
and both
$\tilde{G}(\bar{P}_n(w)+k)$ and $G_{m\vee n}(w,k)$ 
stay bounded when $n \ge 1$ and $k\in \cH$ 
(with $\|k\|_{\cH} \le r$) vary,  we have 
\[
\sup_{ \|k\|_{\cH}\leq r}
\rho_{E} (\tilde{F}(\bar{P}_n(w)+k),F_{m\vee n}(w,k))
\le 
C_{r,w}
\sup_{ \|k\|_{\cH}\leq r}
\rho_{E} (\tilde{G}(\bar{P}_n(w)+k),G_{m\vee n}(w,k)).
\]
As we have seen, the right hand tends to zero as $n\to\infty$.
In same way, we can show 
\[
\sup_{\|k\|_{\cH}\leq r}
\rho_{E} (\tilde{F}(h+k),F_n^{\perp}(w, P_n (h)+k))
\to 0
\qquad
\mbox{as $n\to\infty$.}
\]
These imply \eqref{rhoe2} for $F$.
Thus, we have shown (2).
\end{proof}

The following is the key result of this paper.
Note that twice $\cK$-regular differentiability is 
shown at the level of the path space 
unlike in \cite[Theorem 3.41]{aks}.
Of course, it immediately implies that
$y_t=y_t ({\bf w})$ is twice $\cK$-regularly differentiable for every $t \in [0,T]$,
where $y= (y_t({\bf w}))_{0 \le t \le T}$ the unique solution of 
RDE \eqref{rde.def} with ${\bf x}$ being replaced by 
the Gaussian rough path ${\bf w}$.
\begin{proposition}\label{pr.Kdiff}
Assume {\bf (CYR)}.  
If $V_i$ is of $C_b^{[p] +3}$ for all $0 \le i \le d$, then 
$\Phi \circ 
{\bf L}\colon \cW \to C^{0, p -{\rm var}} ({\mathbb R}^e)$
is twice $\cK$-regularly differentiable.
\end{proposition}

\begin{proof}
We use the same notation as in the proof of 
Proposition \ref{pr.Kreg}, (2).
An unimportant positive constant which depends only 
on the parameter $\star$ is denoted by $C_\star$,
which may vary from line to line.
As is well-known, Fr\'echet-$C^n$ and G\^ateaux-$C^n$
are equivalent (see \cite[Theorem (2.1.27)]{ber} for example). 
Hence, we only consider G\^ateaux derivatives.

We will prove \eqref{rhoe3} for $l=2$ by 
showing that 
\begin{align}
\lefteqn{
\|\tilde{F}(\bar{P}_n(w)+\bullet)-F_{m\vee n}(w,\bullet))\|_{C_{b}^2(B_{\cH}(0,r), E)} 
}\nn\\
&\qquad
\vee \|\tilde{F}(h+\bullet)-F^{\perp}_{n}(w, P_n(h)+\bullet)\|_{C_{b}^2(B_{\cH}(0,r), E)} 
\to 0 
\quad
\mbox{as $n\to\infty$}
\label{conv.200324}
 \end{align} 
 for every $w\in \cA$, $r>0$ and $h \in \cH$.
Convergence of the zeroth order in \eqref{conv.200324}
was already shown 
in the proof of Proposition \ref{pr.Kreg}, (2).

Now we calculate the first order derivative.
For the rest of the proof,  let $r, r'>0$, $w \in {\mathcal A}$, 
$k, l, h \in {\mathcal H}$.
Since $\tilde{F} = \Psi$, we have
\begin{align*}
D_l \tilde{F}(\bar{P}_n(w)+\bullet) \vert_{\bullet =k}
&=
\xi^{[1]}(\bar{P}_n(w)+ k, l)
=
\xi^{[1]}( T_k {\bf w}^{[N_n]}, l)
\end{align*}
Due to the local Lipschitz continuity of $\xi^{[1]}$ we mentioned in 
\eqref{eq.der.Lip.conti}, we have that, 
if $\|k\|_{\cH} \le r$ and $\|l\|_{\cH}\le r'$, then
\begin{align}
\bigl\|
D_l \tilde{F}(\bar{P}_n(w)+\bullet) \vert_{\bullet =k} 
- \xi^{[1]}( T_k {\bf w}, l) \bigr\|_E 
&=
\| \xi^{[1]}( T_k {\bf w}^{[N_n]}, l)
- \xi^{[1]}(T_k{\mathbf L}(w), l) \|_E
\nn\\
&\le 
C_{r, r', w} \rho_{E^\prime} ( {\mathbf w}^{[N_n]}, {\mathbf L}(w))
\to 0
\quad
\mbox{as $n \to \infty$.}
\nn
\end{align}
In particular,
this uniform convergence in $(k, l)$ implies that
\[
D_l  F_{m\vee n}(w,\bullet))\vert_{\bullet =k} 
= D_l   \Phi (T_{\bullet} {\mathbf L}(w))
\vert_{\bullet =k} 
=\xi^{[1]}( T_k {\mathbf L}(w), l) 
\]
and $k \mapsto D  F_{m\vee n}(w,\bullet))\vert_{\bullet =k} 
$ is continuous and
\begin{align}
\lefteqn{
\sup_{\|k\|_{\cH} \le r}
\|D \tilde{F}(\bar{P}_n(w)+\bullet)\vert_{\bullet =k} 
- D F_{m\vee n}(w,\bullet))\vert_{\bullet =k}
\|_{\cH \to E}
}
\nn\\
&\le 
\sup_{\|k\|_{\cH} \le r}\sup_{\|l\|_{\cH} \le 1}
 \| \xi^{[1]}( T_k {\bf w}^{[N_n]}, l)
   - \xi^{[1]}( T_k {\mathbf L}(w), l) \|_E
 \to 0 
 \quad
\mbox{as $n \to \infty$,}
\label{eq.200323-2}
\end{align}
where $\| \,\cdot\, \|_{\cH \to E}$ stands for the 
operator norm for bounded operators from $\cH$ to $E$.

Since $F_n^\perp (w,k) =\Phi (T_k  {\mathbf w}^{*N_n} )
= \Phi (T_{k -w^{[N_k]}}   {\mathbf L}(w) )$,
we can show in the same way as above that 
$D_l F_n^\perp (w,\bullet)\vert_{\bullet =k}
= \xi^{[1]}( T_{k -w^{[N_n]}}  {\mathbf L}(w), l) =
\xi^{[1]}( T_{k}  {\bf w}^{* N_n}, l)$.
Hence, if $\|k\|_{\cH} \le r$ and $\|l\|_{\cH}\le r'$, then
\begin{align}\label{eq.200323-4}
\bigl\|
D_l 
 \tilde{F}(h+\bullet)\vert_{\bullet =k} 
- D_l F^{\perp}_{n}(w, P_n(h)+\bullet) \vert_{\bullet =k}   \bigr\|_E 
&\le
\|  
\xi^{[1]} (T_{h+k} {\bf 0}, l) - \xi^{[1]}( T_{k+P_n(h)}  {\bf w}^{* N_n}, l)
\|_E
\nn\\
&\le 
C_{r, r', w, h} \{
\rho_{E^\prime} ( {\mathbf w}^{*N_n}, {\mathbf 0})
+  \| h- P_n(h)\|_{\cH} \}
\nn\\
&\to
0
\quad
\mbox{as $n \to \infty$.}
\nn
\end{align}
It is also easy to see from this that $k \mapsto 
D F^{\perp}_{n}(w, P_n(h)+\bullet)\vert_{\bullet =k}$ is continuous and 
\begin{equation}
\sup_{\|k\|_{\cH} \le r}
\bigl\|
D
 \tilde{F}(h+\bullet)\vert_{\bullet =k} 
- D F^{\perp}_{n}(w, P_n(h)+\bullet) \vert_{\bullet =k}   \bigr\|_{\cH \to E}
 \to 0 
 \quad
\mbox{as $n \to \infty$.}
\label{eq.200323-5}
\end{equation}
Convergence of the first order derivatives in 
\eqref{conv.200324} follows
immediately from \eqref{eq.200323-2} and \eqref{eq.200323-5}.

Next we calculate the second order derivatives.
We have
\begin{align*}
D^2_{l,l} \tilde{F}(\bar{P}_n(w)+\bullet) \vert_{\bullet =k}
&=
\xi^{[2]}(\bar{P}_n(w)+ k, l)
=
\xi^{[2]}( T_k {\bf w}^{[N_n]}, l).
\end{align*}
Due to the Lipschitz continuity of $\xi^{[2]}$ we mentioned in 
\eqref{eq.der.Lip.conti}, we have that, 
if $\|k\|_{\cH} \le r$ and $\|l\|_{\cH}\le r'$, then
\begin{align}\label{eq.200324-1}
\bigl\|
D^2_{l,l} \tilde{F}(\bar{P}_n(w)+\bullet) \vert_{\bullet =k} 
- \xi^{[2]}( T_k {\bf w}, l) \bigr\|_E 
&=
\| \xi^{[2]}( T_k {\bf w}^{[N_n]}, l)
- \xi^{[2]}({\mathbf L}(w), l) \|_E
\nn\\
&\le 
C_{r, r', w} \rho_{E^\prime} ( {\mathbf w}^{[N_n]}, {\mathbf L}(w))
\to 0
\quad
\mbox{as $n \to \infty$.}
\end{align}
By polarization, we can easily see that, for $\|k\|_{\cH} \le r$ and 
$\|l\|_{\cH}\vee \|\hat{l}\|_{\cH} \le r'$,
\begin{align}%
\bigl\|
4D^2_{l,\hat{l}} \tilde{F}(\bar{P}_n(w)+\bullet) \vert_{\bullet =k} 
-  \{\xi^{[2]}( T_k {\mathbf L}(w), l+\hat{l}) + 
\xi^{[2]}( T_k {\mathbf L}(w), l-\hat{l})\}\bigr\|_E 
\nn\\
\le 
C_{r, r', w} \rho_{E^\prime} ( {\mathbf w}^{[N_n]}, {\mathbf L}(w))
\to 0
\quad
\mbox{as $n \to \infty$.}
\nn
\end{align}
This  uniform convergence in $(k, l)$ 
implies that
\[
D^2_{l,\hat{l}}F_{m\vee n}(w,\bullet))\vert_{\bullet =k} 
= D^2_{l,\hat{l}} \Phi (T_{\bullet} {\mathbf L}(w))
\vert_{\bullet =k} 
=
\frac14 \{\xi^{[2]}( T_k {\mathbf L}(w), l+\hat{l}) + 
\xi^{[2]}( T_k {\mathbf L}(w), l-\hat{l})\}
\]
and $k \mapsto D  F_{m\vee n}(w,\bullet))\vert_{\bullet =k} 
$ is continuous and
\begin{align}
\lefteqn{
\sup_{\|k\|_{\cH} \le r}
\|D^2 \tilde{F}(\bar{P}_n(w)+\bullet)\vert_{\bullet =k} 
- D^2 F_{m\vee n}(w,\bullet))\vert_{\bullet =k}
\|_{\cH \times \cH \to E}
}
\nn\\
&\le 
\frac14
\sup_{\|k\|_{\cH} \le r}\sup_{\|l\|_{\cH} \vee \|\hat{l}\|_{\cH} \le 1}
\{
 \| \xi^{[2]}( T_k {\bf w}^{[N_n]}, l+\hat{l})
   - \xi^{[2]}( T_k {\mathbf L}(w), l+\hat{l}) \|_E
\nn\\
&\qquad +
 \| \xi^{[2]}( T_k {\bf w}^{[N_n]}, l-\hat{l})
   - \xi^{[2]}( T_k {\mathbf L}(w), l-\hat{l}) \|_E
   \}
 \to 0 
 \quad
\mbox{as $n \to \infty$,}
\label{eq.200324-4}
\end{align}
where $\| \,\cdot\, \|_{\cH \times \cH 
\to E}$ stands for the standard
norm for bounded bilinear maps from $\cH\times \cH $ to $E$.

Since $F_n^\perp (w,k) 
= \Phi (T_{k -w^{[N_k]}}   {\mathbf L}(w) )$,
we see that 
$D_l F_n^\perp (w,\bullet)\vert_{\bullet =k}
=
\xi^{[2]}( T_{k}  {\bf w}^{* N_n}, l)$.
Hence, if $\|k\|_{\cH} \le r$ and $\|l\|_{\cH}\vee \|\hat{l}\|_{\cH}\le r'$, then
\begin{align}\label{eq.200324-6}
\lefteqn{
\bigl\|
D^2_{l,\hat{l}}
 \tilde{F}(h+\bullet)\vert_{\bullet =k} 
- D^2_{l,\hat{l}}
 F^{\perp}_{n}(w, P_n(h)+\bullet) \vert_{\bullet =k}   \bigr\|_E 
 }
 \nn\\
&\le
\|  
\xi^{[2]} (T_{h+k} {\bf 0}, l+\hat{l}) - 
\xi^{[2]}( T_{k+P_n(h)}  {\bf w}^{* N_n}, l+\hat{l})
\|_E
  \nn\\
&\qquad  +
     \|  
\xi^{[2]} (T_{h+k} {\bf 0}, l -\hat{l}) - 
\xi^{[2]}( T_{k+P_n(h)}  {\bf w}^{* N_n}, l - \hat{l})
\|_E
\nn\\
&\le 
C_{r, r', w, h} \{
\rho_{E^\prime} ( {\mathbf w}^{*N_n}, {\mathbf 0})
+  \| h- P_n(h)\|_{\cH} \}
\quad 
\to 0
\quad
\mbox{as $n \to \infty$.}
\nn
\end{align}
It is also easy to see from this that 
$k \mapsto 
D^2 F^{\perp}_{n}(w, P_n(h)+\bullet)\vert_{\bullet =k}$ 
is continuous and 
\begin{equation}
\sup_{\|k\|_{\cH} \le r}
\bigl\|
D^2
 \tilde{F}(h+\bullet)\vert_{\bullet =k} 
- D^2 F^{\perp}_{n}(w, P_n(h)+\bullet) \vert_{\bullet =k}   
\bigr\|_{\cH\times \cH \to E}
 \to 0 
 \quad
\mbox{as $n \to \infty$.}
\label{eq.200324-7}
\end{equation}
Convergence of the second order derivatives  in 
\eqref{conv.200324} follows immediately from 
\eqref{eq.200324-4} and \eqref{eq.200324-7}.
This completes the proof.
\end{proof}

%%%%%%%%%%%%%
%\newpage
%%%%%%%%%%
\section{Summary of our result
and comparison with preceding ones}\label{sec.cmpr}

Finally, we summarize our main result 
on positivity of the law of the solution of RDEs in the next theorem.
As in Propositions \ref{pr.Kreg} and \ref{pr.Kdiff}, 
we denote by $y= (y_t({\bf w}))_{0 \le t \le T}$ 
the unique solution of 
RDE \eqref{rde.def} with ${\bf x}$ being replaced by 
the Gaussian rough path ${\bf w}$.
Also,
$\cK=\{K_n\}^{\infty}_{n=1}$ is the same 
non-decreasing, countable exhaustion of $\cH$ as in these propositions.
\begin{theorem}\label{thm.pstv}
Assume {\bf (CYR)} and that 
 $V_i$ is of $C_b^{\infty}$ for all $0 \le i \le d$. 
 Let $t \in (0,T]$ and assume further that $y_t= y_t({\bf w})$
 is non-degenerate in the Malliavin sense \eqref{cond.nondeg}.
 Let $G\colon\cW\rightarrow [0,+\infty)$ be functions which 
 is infinitely differentiable in the sense of the Malliavin calculus.   
 Then, there exists a unique non-negative function
 $f_t\in C^{\infty}(\mathbb{R}^e, {\mathbb R})$ with the properties that $f_t$ has rapidly decreasing derivatives of all orders and 
\[
\int_{\cW}  \phi (y_t({\bf w}) ) G (w)
\mu (dw)
=\int_{\mathbb{R}^e}\phi (z)  f_t (z) dz, 
\qquad
\phi \in C_b(\mathbb{R}^e, \mathbb R).
\]
Assume in addition that $G$ is 
$\cK$-regular with its $\cK$-regularization $\tilde{G}$.
Then,  for $z\in \mathbb{R}^e$, the following  are equivalent:
\begin{itemize}
\item
$f_t (z)>0$.
\item
There exists $h \in \cH$ such that 
$D\Psi (h)_t\colon \cH\rightarrow \mathbb{R}^e$ has rank $e$, 
$\Psi (h)_t=z$ and 
$\tilde{G}(h)>0$. 
\end{itemize}
Here, $\Psi$ is the solution map for the corresponding 
Young ODE \eqref{ode.def}.
 \end{theorem}
 
  \begin{proof}
It is shown in \cite{ina2} that if we assume {\bf (CYR)} and that 
 $V_i$ is of $C_b^{\infty}$, then $y_t({\bf w}) $ is 
 infinitely differentiable in the sense of the Malliavin calculus
 for every $t$. 
The rest of the proof follows immediately 
from  Proposition \ref{pr.Kdiff} and Theorem \ref{aks>0}.
    \end{proof}

 \begin{remark} \label{rem.non-deg}   
 We make a remark on the non-degeneracy of $y_t({\bf w})$,
 which is the key assumption of the above theorem.
 First, we recall H\"ormander's bracket generating condition 
 on $V_i$'s.
 Set $\cV_1 =\{ V_1, \ldots, V_d\}$ and 
 $\cV_k =\{ [V_i, W ] \colon 0 \le i \le d, W \in \cV_{k-1} \}$
 for $k \ge 2$.
We say $\{V_0, V_1,\cdots, V_d\}$ satisfies 
the H\"ormander condition at the starting point 
$a\in \mathbb{R}^e$ if 
$\{ W(a) \colon W \in \cup_{k=1}^\infty \cV_{k}\}$ spans
$\mathbb{R}^e$.

Non-degeneracy of $y_t({\bf w})$ under this condition 
at the starting point $a$ was proved in \cite{chlt, got1}
for a class of Gaussian processes,
wherein the assumption on $w$ is stronger than {\bf (CYR)},
but examples still include fBM with $H \in (1/4, 1/2]$. 
     \end{remark}

\begin{remark}\label{rem.kikaku}
Let us compare Theorem \ref{thm.pstv} above with 
two preceding results \cite{bnot, got2},
in which $w$ is fBM with $H \in (1/4, 1/2]$.
In this case, we have Remark \ref{rem.non-deg} above
for the non-degeneracy assumption in Theorem \ref{thm.pstv}.
\begin{itemize}
\item[(1)]~Suppose that $\{ V_i \colon 1 \le i \le d\}$ satisfies 
the ellipticity condition everywhere, i.e., 
$\{V_i (z) \colon 1 \le i \le d \}$ spans 
$\mathbb{R}^e$ at every $z \in \mathbb{R}^e$.
In this case, $D\Psi (h)_t\colon \cH\rightarrow \mathbb{R}^e$ has rank $e$ for every $h\in \cH$.
By standard argument, one can easily find for every $z$
a $C^1$-path $h$ such that $\Psi (h)_t=z$. 
Since a $C^1$-path belongs to Cameron-Martin space
for fBM, we can use Theorem \ref{thm.pstv}
to conclude that 
$f_t (z) >0$ for all $z\in\mathbb{R}^e$ and $t >0$ when 
$G \equiv 1$.
Therefore, Theorem \ref{thm.pstv} implies the positivity 
result in \cite[Theorem 1.4]{bnot}.
\item[(2)]~
We explain that Theorem \ref{thm.pstv} implies 
the positivity result in \cite{got2}, in which $V_0 \equiv 0$
and $\{ V_i \colon 1 \le i \le d\}$ is assumed to  satisfy
the uniform H\"ormander condition on $\mathbb{R}^e$.
By using Sard's theorem and Chow-Radevski's theorem,
they proved in \cite{got2} that, for every $z\in\mathbb{R}^e$,
there exists $h \in \cH$ such that 
$D\Psi (h)_t\colon \cH\rightarrow \mathbb{R}^e$ has rank $e$, 
$\Psi (h)_t=z$.
(This part seems non-trivial.)
So, we can use Theorem \ref{thm.pstv}
to recover the positivity result in \cite{got2}, namely, 
$f_t (z) >0$ for all $z\in\mathbb{R}^e$ and $t >0$ when 
$G \equiv 1$.
\item[(3)]~Other reasons why Theorem \ref{thm.pstv}
are more general than the results in \cite{bnot, got2}
are as follows:
\begin{itemize}
\item[(a)]~In \cite{bnot, got2}, only the case $G \equiv 1$ is considered.
\item[(b)]~In \cite{bnot, got2}, the Gaussian process $w$ 
is restricted to fBM with $H \in (1/4, 1/2]$ only. 
\item[(c)]~
The results in \cite{bnot, got2} prove
everywhere-positivity for the density.
The case $f_t (z) =0$ was not studied there unlike in Theorem \ref{thm.pstv} above.
\end{itemize}
\end{itemize}
\end{remark}   
%
%

%%%%%%%%%%%%%%%%%%%%%%%%%%%%%%%%
%%%%%%%%% Acknowledgement %%%%%%%%%%%
%%%%%%%%%%%%%%%%%%%%%%%%%%%%%%%%%%%%%
\noindent
{\bf Acknowledgement:}~
The first named author
 was partially supported by JSPS KAKENHI (JP20H01807) and Grant-in-Aid for JSPS Fellows (JP18F18314).
The second named author was partially supported by the NSF of China (11802216), the China Postdoctoral
Science Foundation (2019M651334) and JSPS Grant-in-Aid for JSPS Fellows (JP18F18314). 

%%%%%%%%%%%%%%%%%%%%%%%%%%%%%%%%%%%%
%%%%%%%%%%%%%%%%%%%%%%%%%%%%%%%%%%%%%%%%%%%%%%%%%%%%%%%%
%%%%%%%%%%%%%%%%%%%%%%%%%%%%%%
%  references
%%%%%%%%%%%%%%%%%%%%%%%%%%%%%%%%%%%%%%%%%%%%%%%%%%%%%%%%%%%%%%%%%%%%%%%%%

%%%%%%%%%%%%%%%%%%%%%%%%%%%%%%%%%%%%
\bigskip
%%%%%%%%%%%%%%%%%%%%%%%%%%%%%%%%%%%%
%%%%%%%%%% Affiliations %%%%%%%%%%%%
%%%%%%%%%%%%%%%%%%%%%%%%%%%%%%%%%%%%
\begin{flushleft}
  \begin{tabular}{ll}
    Yuzuru \textsc{Inahama}
    \\
    Faculty of Mathematics,
    Kyushu University,
    \\
    744 Motooka, Nishi-ku, Fukuoka, 819-0395, Japan.
    \\
    Email: {\tt inahama@math.kyushu-u.ac.jp}
  \end{tabular}
\end{flushleft}

\begin{flushleft}
  \begin{tabular}{ll}
    Bin \textsc{Pei}
    \\
School of Mathematical Sciences,
    Fudan University, 
        \\
        Shanghai, 200433, China
    \\
    \quad and 
    \\
       Faculty of Mathematics,
    Kyushu University,
    \\
    744 Motooka, Nishi-ku, Fukuoka, 819-0395, Japan.    
    \\
    Email: {\tt binpei@hotmail.com}
  \end{tabular}
\end{flushleft}
%%%%%%%%%%%%%%%%%%%%%%%%%%%%%%%%%%%


\begin{thebibliography}{00}

\bibitem{aks}
Aida, S.; Kusuoka, S.; Stroock, D.; 
On the support of Wiener functionals. 
Asymptotic problems in probability theory: Wiener functionals and asymptotics (Sanda/Kyoto, 1990), 3--34, 
Pitman Res. Notes Math. Ser., 284, Longman Sci. Tech., Harlow, 1993. 



\bibitem{bnot}
Baudoin, F.; Nualart, E.; Ouyang, C.; Tindel, S.;
On probability laws of solutions to differential systems driven by a fractional Brownian motion. 
Ann. Probab. 44 (2016), no. 4, 2554--2590. 


\bibitem{bl2}
Ben Arous, G.; L\'eandre, R.;
D\'ecroissance exponentielle du noyau de la chaleur sur la diagonale. II. 
Probab. Theory Related Fields 90 (1991), no. 3, 377--402. 


\bibitem{ber}
Berger, M. S.;
Nonlinearity and functional analysis. 
Academic Press, New York-London, 1977.
 


\bibitem{chlt}
Cass, T.; Hairer, M.; Litterer, C.; Tindel, S.;
Smoothness of the density for solutions to Gaussian rough differential equations. 
Ann. Probab. 43 (2015), no. 1, 188--239. 

 
%\bibitem{fggr} 
%Friz, P. K.; Gess, B.; Gulisashvili, A.; Riedel, S.;
%The Jain-Monrad criterion for rough paths and applications to random Fourier series and non-Markovian H\"{o}rmander theory. 
%Ann. Probab. 44 (2016), no. 1, 684--738. 
 
 
 
\bibitem{fvbk}
Friz, P. K.; Victoir, N. B.; 
Multidimensional stochastic processes as rough paths.
Cambridge University Press, Cambridge, 2010.

\bibitem{got2}
Geng, X.; Ouyang, C.; Tindel, S.;
Precise local estimates for hypoelliptic differential equations driven by fractional Brownian motions.
Preprint. arXiv:1907.00171.


\bibitem{got1}
Gess, B.; Ouyang, C.; Tindel, S.;
Density bounds for solutions to differential equations driven by Gaussian rough paths.
Preprint. arXiv:1712.02740.



\bibitem{ina2}
Inahama, Y.; 
Malliavin differentiability of solutions of rough differential equations. 
J. Funct. Anal. 267 (2014), no. 5, 1566--1584. 


\bibitem{ina1}
Inahama, Y.; 
Short time kernel asymptotics for Young SDE by means of Watanabe distribution theory. 
J. Math. Soc. Japan 68 (2016), no. 2, 535--577. 



\bibitem{ll}
Li, X.-D.; Lyons, T. J.; 
Smoothness of It\^{o} maps and diffusion processes on path spaces, I, 
Ann. Sci. \'{E}cole Norm. Sup. (4), 39 (2006), no. 4, 649--677.

\bibitem{gross1}
Gross, L.; Abstract Wiener Spaces. Berkeley Symp. on Math. Statist. and Prob.
Proc. Fifth Berkeley Symp. on Math. Statist. and Prob., Vol. 2, Pt. 1 (Univ. of Calif. Press, 1967), 31-42
\end{thebibliography}
\end{document}